\documentclass[a4paper,10pt]{amsart}
\usepackage{amssymb,hyperref}

\title{Congruence testing for odd subgroups of the modular group}
\author{Thomas Hamilton}
\address[Hamilton]{Premier Pensions Management \\ Corinthian House \\ 17 Lansdowne Road \\ Croydon CR0 2BX, UK }

\author{David Loeffler}
\address[Loeffler]{Mathematics Institute \\ University of Warwick \\ Coventry CV4 7AL, UK}
\email{d.a.loeffler@warwick.ac.uk}

\newtheorem{theorem}{Theorem}
\newtheorem{proposition}[theorem]{Proposition}
\newtheorem*{question}{Question}

\begin{document}

\begin{abstract}
 We give a computationally effective criterion for determining whether a finite-index subgroup of $\operatorname{SL}_2(\mathbf{Z})$ is a congruence subgroup, extending earlier work of Hsu for subgroups of $\operatorname{PSL}_2(\mathbf{Z})$.
\end{abstract}

\maketitle

Recall that a finite-index subgroup of $\operatorname{SL}_2(\mathbf{Z})$ is said to be a \emph{congruence subgroup} if it is defined by congruence conditions on the entries of its elements; formally, a subgroup is congruence if it contains the subgroup $\Gamma(N)$ of matrices congruent to the identity modulo $N$, and the least such $N$ is its \emph{level}.

We are interested in the following question:

\begin{question}
 Is there an efficient procedure that will determine whether a finite-index subgroup of $\operatorname{SL}_2(\mathbf{Z})$ is congruence?
\end{question}

One such algorithm follows from the following theorem, proved in \cite{ksv}, which is an extension of a classical theorem of Wolfahrt:

\begin{theorem}[Kiming--Sch\"utt--Verrill]\label{thm:ksv}
 Let $\Gamma \le \operatorname{SL}_2(\mathbf{Z})$ and let $d$ be the lowest common multiple of the widths of the cusps of $\Gamma$. If $\Gamma$ is congruence, then its level is either $d$ or $2d$.
\end{theorem}

(The case of level $2d$ can only occur if $\Gamma$ is \emph{odd}, i.e.~does not contain $-1$.)

In principle, one can now determine whether $\Gamma$ is congruence by calculating explicitly a list of generators for $\Gamma(N)$, where $N = d$ or $2d$ as appropriate, and testing whether each of these is contained in $\Gamma$. This approach is used in \emph{op.cit.} in order to give explicit examples of non-congruence lifts to $\operatorname{SL}_2(\mathbf{Z})$ of congruence subgroups of $\operatorname{PSL}_2(\mathbf{Z})$. However, the number of generators of $\Gamma(N)$ grows rather quickly with $N$, so this algorithm rapidly becomes impractical for large values of $N$.

We present the following alternative approach to the above problem. As has been noted by Hsu \cite{hsu} and others, a convenient data structure for representing a subgroup of $\operatorname{SL}_2(\mathbf{Z})$ of index $m$ is by the homomorphism $\operatorname{SL}_2(\mathbf{Z}) \to S_m$ given by left multiplication on the cosets $\operatorname{SL}_2(\mathbf{Z}) / \Gamma$. This, in turn, can be represented by two permutations giving the action of the generators $L = \begin{pmatrix} 1 & 0 \\ 1 & 1 \end{pmatrix}$ and $R = \begin{pmatrix} 1 & 1 \\ 0 & 1 \end{pmatrix}$ of $\operatorname{SL}_2(\mathbf{Z})$ on the cosets $\operatorname{SL}_2(\mathbf{Z}) / \Gamma$.

The computer algebra package Sage contains a library of routines for working with subgroups defined in this way, implemented by Vincent Delecroix and the second author based on an earlier implementation by Chris Kurth.

\begin{theorem}\label{main}
 Let $N = d$ if $-1 \in \Gamma$ and $N = 2d$ otherwise. Then there exists an explicit list of relations $\mathcal{L}_N$ in $L$ and $R$ (of length $\le 7$), such that $\Gamma$ is congruence if and only if the permutation representation of $\operatorname{SL}_2(\mathbf{Z})$ corresponding to $\Gamma$ satisfies the relations in $\mathcal{L}_N$.
\end{theorem}

This theorem has been proved for subgroups containing $-1$ by Hsu \cite{hsu}; our proof follows Hsu's closely, except that we use the Kiming--Sch\"utt--Verrill theorem (Theorem \ref{thm:ksv}) in place of the classical theorem of Wolfahrt.

\begin{proposition}
 Let $N \ge 1$. There is an explicit finite list $\mathcal{L}_N$ of words in $L$ and $R$ which \emph{normally generates} $\Gamma(N)$ (that is, $\Gamma(N)$ is the smallest normal subgroup of $\operatorname{SL}_2(\mathbf{Z})$ containing the elements in $\mathcal{L}_N$).
\end{proposition}

\begin{proof}
 See \cite[Lemmas 3.3-3.5]{hsu}. (This is essentially the same problem as determining a set of relations for a presentation of $\operatorname{SL}_2(\mathbf{Z}/ N\mathbf{Z})$ in terms of the images of $L$ and $R$.)
\end{proof}

\begin{proof}[Proof of Theorem \ref{main}]
 Let $N$ be as defined in the statement of the theorem. We know that $\Gamma$ is congruence if and only if it contains $\Gamma(N)$. Let $\Gamma'$ be the \emph{normal core} of $\Gamma$, i.e.~the intersection of the conjugates of $\Gamma$ in $\operatorname{SL}_2(\mathbf{Z})$; then, since the elements of $\mathcal{L}_N$ normally generate $\Gamma(N)$, it follows that $\Gamma$ is congruence if and only if $\mathcal{L}_N \subset \Gamma'$.

 However, $\Gamma'$ is precisely the kernel of the map $\phi: \operatorname{SL}_2(\mathbf{Z}) \to S_m$ giving the permutation representation of $\Gamma$. So $\Gamma$ is congruence if and only if $\phi$ is trivial on the elements of $\mathcal{L}_N$.
\end{proof}

(One could clearly adapt this argument to work with other explicit descriptions of $\Gamma$ as long as one has an algorithm for computing whether a given element of $\operatorname{SL}_2(\mathbf{Z})$ lies in the normal core of $\Gamma$.)

We now reproduce, for the reader's convenience, an explicit list of relations $\mathcal{L}_N$ as in Theorem \ref{main} (which are almost identical to those appearing in \cite[Theorem 3.1]{hsu}).

\begin{itemize}
 \item If $N$ is odd, one may take $\mathcal{L}_N$ to contain the single relation
 \[ (R^2 L^{-\tfrac{1}{2}})^3 = 1,\]
 where $\tfrac12$ is the multiplicative inverse of $2 \bmod N$. (This case can, of course, only occur if $-1 \in \Gamma$ and is thus identical to the first case of Hsu's theorem.)
 \item If $N$ is a power of 2, let $S = L^{20} R^{\tfrac15} L^{-4} R^{-1}$, where $\tfrac15$ is the multiplicative inverse of $5 \bmod N$. Then one may take $\mathcal{L}_N$ to consist of the three relations
 \begin{align*}
  (L R^{-1} L)^{-1} S (LR^{-1} L) &= S^{-1}, \\
  S^{-1} R S &= R^{25}, \\
  (SR^5 LR^{-1} L)^3 &= (LR^{-1} L)^2.
 \end{align*}
 (Note that if we assume that $-1 \in \Gamma$ we may replace the last relation with $(SR^5 LR^{-1} L)^3 = 1$, which is the relation appearing in Hsu's paper, but for odd subgroups we must use the slightly more complicated relation above.)
 \item If $N = em$ where $e$ is a power of 2, $m$ is odd and $e, m > 1$, then let $c, d$ be the unique integers mod $N$ such that $c = 0 \bmod e, c = 1 \bmod m$, $d = 1 \bmod e$, $d = 0 \bmod m$. Write $a = L^c, b = R^c, l = L^d, r = R^d$ and $s = l^{20} r^{\tfrac 1 5} l^{-4} r^{-1}$, where $\tfrac 1 5$ is interpreted mod $m$. Then we may take $\mathcal{L}_N$ to consist of the seven elements
 \begin{align*}
  [a, r] &= 1, \\
  (a b^{-1} a)^4 &=1, \\
  (ab^{-1}a)^2 &= (b^{-1} a)^3, \\
  (a b^{-1} a)^2 &= (b^2 a^{-\tfrac12})^3,\\
  (lr^{-1}l)^{-1} s (lr^{-1}l) &= s^{-1},\\
  s^{-1} r s &= r^{25},\\
  (lr^{-1} l)^2 &= (s r^5 l r^{-1} l)^3.
 \end{align*}
\end{itemize}

\subsection*{Acknowledgements} This paper is a much-condensed version of the first author's University of Warwick MMath dissertation, written in 2011-12 under the supervision of the second author. We are grateful to Vincent Delecroix for the original observation that Hsu's test should generalize to odd subgroups.

\providecommand{\bysame}{\leavevmode\hbox to3em{\hrulefill}\thinspace}
\providecommand{\MR}[1]{}
\renewcommand{\MR}[1]{%
 MR \href{http://www.ams.org/mathscinet-getitem?mr=#1}{#1}.
}
\providecommand{\href}[2]{#2}
\newcommand{\articlehref}[2]{\href{#1}{#2}}

\end{document}